\documentclass[11pt]{amsart}
\usepackage{graphicx}
\usepackage{amssymb}
\usepackage{cite}
\usepackage{hyperref}
\usepackage{enumerate}
\usepackage{txfonts} 
\usepackage{mathtools}
\usepackage[usenames,dvipsnames]{color}

\numberwithin{equation}{section} 

\newtheorem{theorem}{Theorem}[section]
\newtheorem{theorem*}{Theorem}
\newtheorem{lemma}[theorem]{Lemma}
\newtheorem{proposition}[theorem]{Proposition}
\newtheorem{corollary}[theorem]{Corollary}

\newtheorem{remark}[theorem]{Remark}
\newtheorem{remark*}[theorem*]{Remark}
\newtheorem{definition}[theorem]{Definition}

\def\R{{\mathbb R}}

\def\S{{\mathbb S}}

\def\cE{{\mathcal E}}
\def\cD{{\mathcal D}}
\def\cB{{\mathcal B}}
\def\cK{{\mathcal K}}

\def\1{\left(}
\def\2{\right)}
\def\3{\left\{}
\def\4{\right\}}
\def\8{\infty}
\def\sm{\setminus}
\def\ss{\subseteq}

\begin{document}
\title{Interior regularity for fractional systems }

\author[L. Caffarelli]{Luis Caffarelli}
\address{The University of Texas at Austin, Department of Mathematics, 2515 Speedway,
Austin, TX 78751, USA
}
\email{caffarel@math.utexas.edu}

\author[G. D\'avila]{Gonzalo D\'avila}
\address{Universidad T\'ecnica Federico Santa Mar\'ia, Departamento de Matem\'atica, Avenida Espa\~na 1680, Valpara\'iso, Chile}
\email{gonzalo.davila@usm.cl}

\begin{abstract}
We study the regularity of solutions of elliptic fractional systems of order $2s$, $s\in(0,1)$, where the right hand side $f$ depends on a nonlocal gradient and has the same scaling properties as the nonlocal operator. Under some structural conditions on the system we prove interior H\"older estimates in the spirit of \cite{Caffarelli-Systems}. Our results are stable in $s$ allowing us to recover the classic results for elliptic systems  due to S. Hildebrandt and K. Widman \cite{Hildebrandt-Widman} and M. Wiegner \cite{Wiegner}.
\end{abstract}

\maketitle

%%%%%%%%%%%%%%%%%%%%%%%%%%%%%%%%%%%%%%%%%%%%%%%%%%%%%%%%%%%%%%%%%%%%%%%%%%%%%%%%%%%%%%%%%%%%%%%%%%%%%%%%%%%%%%%%%%%%%%%%
%%%%%%%%%%%%%%%%%%%%%%%%%%%%%%%%%%%%%%%%%%%%%%%%%Introduction%%%%%%%%%%%%%%%%%%%%%%%%%%%%%%%%%%%%%%%%%%%%%%%%%%%%%%%%%%%
%%%%%%%%%%%%%%%%%%%%%%%%%%%%%%%%%%%%%%%%%%%%%%%%%%%%%%%%%%%%%%%%%%%%%%%%%%%%%%%%%%%%%%%%%%%%%%%%%%%%%%%%%%%%%%%%%%%%%%%%

\section{Introduction}

In this work we are interested in the interior regularity of bounded weak or viscosity solutions of fractional systems of the form 
\[
-Lu_i=f_i(x,u,\cB_K(u,u)),  \ x\in\Omega, \ 1\leq i\leq m,
\]
where $\Omega\ss\R^n$ and $u=(u_1,...,u_m)$ is a vector valued function. Given $s\in(0,1)$ and a real-valued function $v$ our nonlocal operator $L$ is given by
\[
L_K(v)=\int_{\R^n}(v(y)-v(x))K(y-x)dy,
\]
where the kernel $K$ is assumed to be symmetric and to satisfy 
\[
(1-s)\frac{\lambda}{|y|^{n+2s}}\leq K(y) (1-s)\frac{\Lambda}{|y|^{n+2s}}.
\]
Finally $\cB_K$ stands for 
\[
\cB_K(u,w)=\frac{(1-s)c_n}{2}\int_{\R^n}(u(x)-u(y))\cdot(w(x)-w(y))K(y)dy,
\]
and is a fractional derivative of order $s$ playing the role of $\nabla u \nabla w$. This operator appears naturally when studying fractional harmonic maps to the sphere. We refer to Section \ref{sectionharmonicsphere} for more details on this operator and fractional harmonic maps into the sphere.

In the local case, the interior regularity results are due to S. Hildebrandt and K. Widman \cite{Hildebrandt-Widman} and M. Wiegner \cite{Wiegner}. More precisely they studied the regularity of weak solutions to systems of the form
\[
-D_\beta\left[A^{\alpha\beta}(x,u,\nabla u)D_\alpha u_i\right]=f_i(x,u,\nabla u),
\]
with a right hand side satisfying
\begin{itemize}
\item $|f(x,u,\nabla u)|\leq a|\nabla u|^2+b(x)$. (quadratic growth)
\item $u\cdot f(x,u,p)\leq a^*|p|^2+b^*(x)$.
\end{itemize}
for a pair of functions $b, b^*\in L^q$ with $q>n/2$.

They prove interior H\"older estimates under the structural assumption
\[
aM+a^*<2,
\] 
where $M=\left\|u\right\|_\infty$. The structural condition is necessary as the classic harmonic map with values into the sphere $u:=x/|x|:\R^n\to\S^{n-1}$ solves 
\[
(-\Delta)u=u|\nabla u|^2,
\]
and fails to be regular. Note that in this particular case the structural condition reads as $aM+a^*=2$, that is, we are outside the feasible range. 

The proof of S. Hildebrandt and K. Widman \cite{Hildebrandt-Widman} and M. Wiegner \cite{Wiegner} relies heavily on harmonic analysis and the use of the Green functions associated to the linear operators. The structural condition is somewhat hidden in the proof and there is little geometric insight on it. Also, since their proof relies on the divergence structure of the system it does not apply directly to viscosity solutions of systems of the form 
\begin{align}\label{nondiv}
a_{ij}u_{ij}=f(x,u,\nabla u).
\end{align}

A few years later this result was proved by L. Caffarelli in \cite{Caffarelli-Systems} with a completely different strategy. Here the main idea was to control the oscillation of the function using an auxiliary scalar equation. In his proof the structural condition has a geometric interpretation; when satisfied the solution of the system becomes a contraction mapping. This naturally allows to control the oscillation of the function and a standard iterative argument then leads to H\"older continuity of the solution. Moreover, even though not explicitly stated, the proof works also in the nondivergence setting \eqref{nondiv} and therefore the regularity results apply to viscosity solutions.

The paper is organized as follows. Some standard notation and definitions are given in Section \ref{sectiondef}. We also state our main result and the necessary hypotheses. In Section \ref{sectionharmonicsphere} we discuss briefly fractional harmonic maps into spheres. This is our main motivation for studying systems where the right hand side is dominated by the nonlocal operator $\cB$, since it plays the role of the gradient squared of classic harmonic maps. Then, in Section \ref{sectionnonlocalsystems} we prove our result in the spirit of \cite{Caffarelli-Systems}. Finally, Section \ref{sectionextension} is dedicated to extend the result to linear operators with symmetric kernels and discuss the passage to the limit $s\to1$. In this section we also give a counterexample of regularity when the structural condition is not satisfied.

%%%%%%%%%%%%%%%%%%%%%%%%%%%%%%%%%%%%%%%%%%%%%%%%%%%%%%%%%%%%%%%%%%%%%%%%%%%%%%%%%%%%%%%%%%%%%%%%%%%%%%%%%%%%%%%%%%%%%%%%
%%%%%%%%%%%%%%%%%%%%%%%%%%%%%%%%%%%%%%%%%%%%%%%%%Definitions and notation%%%%%%%%%%%%%%%%%%%%%%%%%%%%%%%%%%%%%%%%%%%%%%%
%%%%%%%%%%%%%%%%%%%%%%%%%%%%%%%%%%%%%%%%%%%%%%%%%%%%%%%%%%%%%%%%%%%%%%%%%%%%%%%%%%%%%%%%%%%%%%%%%%%%%%%%%%%%%%%%%%%%%%%%
\section{Definitions and statement of the main result}\label{sectiondef}

Let us introduce some standard notation and the notion of weak and viscosity solution. We point out that our method does not rely on the divergence form of the equation but rather on the Harnack inequality for nonlocal operators, which is available for both notions of solutions. Furthermore, since we are looking for a priori bounds, a standard regularization procedure allows us to assume that solutions are smooth.

We denote by $\cK(\lambda, \Lambda)$ the family of kernels $K$ such that $K(y)=K(-y)$ and
\begin{align}\label{boundK}
(1-s)\frac{\lambda}{|y|^{n+2s}}\leq K_i(y) (1-s)\frac{\Lambda}{|y|^{n+2s}},
\end{align}
with $s\in(0,1)$. When there is no chance of confusion we will simply denote $\cK(\lambda,\Lambda)$ by $\cK$. Given a kernel $K\in\cK$ we denote by $L_K$ the linear nonlocal operator given by
\[
L_K(v)=\text{P.V.}\int_{\R^n}(v(y)-v(x))K(y-x)dy,
\]
The last integral is well defined whenever the function $v$ is punctually $C^{1,1}$ and has an integrable tail. Note that since the kernel is symmetric the operator can be written as
\[
L_K(v)=\frac{1}{2}\int_{\R^n}(v(x+y)+v(x-y)-2v(x))K(y)dy.
\]

Given two smooth bounded maps $u,w:\R^n\to\R^m$ we define the bilinear form $\cB_K$ associated to the kernel $k\in\cK$, as in the introduction, by
\[
\cB_K(u,w)=\frac{1}{2}\int_{\R^n}(u(x)-u(y))\cdot(w(x)-w(y))K(y)dy,
\]
In the special case of the fractional Laplacian
\[
(-\Delta)^su=(1-s)c_n\text{P.V.}\int_{\R^n}\frac{u(x)-u(y)}{|x-y|^{n+2s}}dy
\]
we will denote it's associated bilinear form just by $\cB$. Here the constant $c_n$ is chosen such that $(-\Delta)^su\to(-\Delta)u$. For more details on this we refer to Section \ref{sectionharmonicsphere}.

Let $\Omega\subset\R^n$ be an open bounded set, we are interested in weak and/or viscosity solutions to the following system 
\begin{align}\label{generalsystem}
-L u_i=f_i(x,u,\cB_K(u,u)), \ x\in\Omega.
\end{align}
Here $f_i:\R^n\times\R^m\to\R$ is a bounded function and we will usually denote by $f=(f_1,...,f_m)$ the associated map. We define now weak solutions.
\begin{definition}
A map  $u=(u_i)\in  H^s(\Omega;\R^m)\cap L^\infty(\R^n)$ is a weak solution to the nonlocal system \eqref{generalsystem} if for all test maps $\varphi$ we have 
\[
\iint_{\R^{2n}}(v(x)-v(y))(\varphi(x)-\varphi(y))K(y-x)dydx=\int_{\R^n}f_i(x,u,\cB_K(u,u))\varphi(x)dx,
\]
for all $i=1,...,m$.
\end{definition}
Now, we say that $\varphi$ touches $u$ by above (below) in a neighborhood $N$ if
\begin{itemize}
\item $\varphi(x)=u(x)$ for $x\in\Omega$.
\item $\varphi(y)>u(y)$ ($\varphi(y)<u(y)$) in a neighborhood $N$.
\end{itemize}
We can define now viscosity solutions.
\begin{definition}
A function $u:\R^n\to\R^m$ upper (lower) semicontinuous in $\bar\Omega$ is a viscosity subsolution (supersolution) to \eqref{generalsystem}if anytime a smooth map $\varphi$ touches $u$ by above (below) then 
\[
u_\varphi(x)=\begin{cases} \varphi(x) &\mbox{if } x\in\Omega \\ 
u(x) & \mbox{if } x\in\R^n\sm\Omega \end{cases}
\]
satisfies
\[
Lu_\varphi\leq (\geq) f(x,u_\varphi,,\cB_K(u_\varphi,u_\varphi)).
\]
A viscosity solution $u$ is both a viscosity subsolution and supersolution.
\end{definition}
In order to obtain regularity for solutions to \eqref{generalsystem} we need to impose some growth conditions on the right hand. The hypotheses needed are the following.
\begin{itemize}
\item[(H1.1)] {\it Small $2s$ growth} There are constants $a$ and $b$ such that
\[
|f(x,u(x),\cB(v(y),v(y)))|\leq a\cB(v(y),v(y))+b
\]
for all smooth maps $v:\R^n\to\R^m$ and $x,y\in \Omega$. 
\item[(H1.2)] There are constants $a^*$ and $b^*$ such that
\[
u(x)\cdot f(x,u(x),\cB(v(y),v(y)))\leq a^*\cB(v(y),v(y))+b^*,
\]
 For all smooth maps $v:\R^n\to\R^m$ and $x,y\in \Omega$.
\item[(H2)] $\|u\|_{L^\infty(\Omega)}\leq M$.
\end{itemize}
Note that hypotheses (H1) and (H1.1) are the nonlocal analogous to the conditions imposed in \cite{Hildebrandt-Widman} and \cite{Caffarelli-Systems}.

We point out that the size of the solution plays a relevant role in the regularity. In the local, case harmonic maps into the unitary sphere are not regular ($M=1$) but they are when the target domain is  some compact subset of the open ball as seen in the paper by S. Hildebrandt, H. Kaul and K. Widman \cite{H-K-W}. 

Now we are in shape to state our main result.
\begin{theorem}\label{maintheorem}
Let $u$ be a weak or viscosity solution to \eqref{generalsystem} and assume that hypotheses H1, H1.1 and H2 hold. Furthermore assume that the structural condition $aM+a^*<2$, then there exists $\alpha$ depending only on $\lambda$, $\Lambda$ and the dimension $n$ such that $u\in C^\alpha(\Omega';\R^m)$ where $\Omega'$ is any open such that $\bar\Omega'\subset\Omega$. 
\end{theorem}

%%%%%%%%%%%%%%%%%%%%%%%%%%%%%%%%%%%%%%%%%%%%%%%%%%%%%%%%%%%%%%%%%%%%%%%%%%%%%%%%%%%%%%%%%%%%%%%%%%%%%%%%%%%%%%%%%%%%%%%%
%%%%%%%%%%%%%%%%%%%%%%%%%%%%%%%%%%%%%%%%%%%%%%%%%Fractional harmonic maps%%%%%%%%%%%%%%%%%%%%%%%%%%%%%%%%%%%%%%%%%%%%%%%
%%%%%%%%%%%%%%%%%%%%%%%%%%%%%%%%%%%%%%%%%%%%%%%%%%%%%%%%%%%%%%%%%%%%%%%%%%%%%%%%%%%%%%%%%%%%%%%%%%%%%%%%%%%%%%%%%%%%%%%%

\section{Fractional harmonic maps into the sphere}\label{sectionharmonicsphere}

In \cite{Millot-Sire} V. Millot and Y. Sire studied solutions to the Ginzburg-Landau system
\begin{align*}
(-\Delta)^{1/2}v_\varepsilon&=\frac{1}{\varepsilon}(1-|v_\varepsilon|^2) \ \ \text{in }\Omega\\
v_\varepsilon&=g, \ \ \text{in }\R^n\sm\Omega,
\end{align*}
where $g:\R^n\to\R^m$ is a smooth function with $|g|=1$ in $\R^n\sm\Omega$ and $\Omega\subset\R^n$ be a open bounded set. Among several properties of such solutions they studied the limit as $\varepsilon\to0$ and proved that they converged weakly to sphere valued 1/2-harmonic maps. The limiting equation had a right hand side that involves the bilinear form $\cB$ of order 1/2, as introduced in the previous section. 

Following the construction done by V. Millot and Y. Sire in the case of fractional 1/2 maps, we study fractional harmonic maps into spheres of general order $s$. We will also borrow the notation from \cite{Millot-Sire}.

Let $\Omega\subset\R^n$ be a open bounded set and $u\R^n\to\R^m$ a smooth function. We define the fractional Laplacian of order $s$ by 
\[
(-\Delta)^s u=(1-s)c_{n}\text{P.V.}\int_{\R^n}\frac{u(x)-u(y)}{|x-y|^{n+2s}}dy,
\]
where $c_{n}$ is a normalizing constant such that for smooth functions $(-\Delta)^s\to(-\Delta)$ (see for examples \cite{Hitchhikers}). We will usually refer to $c_{n,s}=(1-s)c_n$. Now we can define the action of the operator by
\begin{align*}
\langle (-\Delta)^s u,\varphi \rangle_\Omega=&\frac{c_{n,s}}{2}\iint_{\Omega\times\Omega}\frac{(u(x)-u(y))(\varphi(x)-\varphi(y))}{|x-y|^{n+2s}}dxdy\\
&+c_{n,s}\iint_{\Omega\times(\R^n\sm\Omega)}\frac{(u(x)-u(y))(\varphi(x)-\varphi(y))}{|x-y|^{n+2s}}dxdy,
\end{align*}
where $\varphi\in\cD(\Omega, \R^m)$. Define the $s$-energy $\cE_s$ by
\begin{align}\label{energy}
\cE_s(u,\Omega)=&\frac{c_{n,s}}{4}\iint_{\Omega\times\Omega}\frac{|u(x)-u(y)|^2}{|x-y|^{n+2s}}dxdy\\
&+\frac{c_{n,s}}{2}\iint_{\Omega\times(\R^n\sm\Omega)}\frac{|u(x)-u(y)|^2}{|x-y|^{n+2s}}dxdy,
\end{align} 
and note that the action by the fractional operator defines a distribution on $\Omega$ when 
\[
\cE_s(u,\Omega)<\infty,
\]
and $u\in L^2_{loc}(\R^n;\R^m)$. In this case we say that $u$ is admissible and $(-\Delta)^s$ belongs to $H^{-s}(\Omega;\R^m)$. For more details on fractional Sobolev spaces and the action of the fractional Laplacian we refer to \cite{Hitchhikers}.

We can introduce now the notion of $s$-harmonic maps with values into the sphere.
\begin{definition}
Let $u\in \hat H^s(\Omega;\R^m)=\{u\in L^2_{loc}(\R^n;\R^m) \text{ s.t. } \cE_s(u,\Omega)<\infty \}$ be such that $|u|=1$ a.e. in $\Omega$. We say that $u$ is weakly $s$-harmonic into $\S^{m-1}$ in $\Omega$ if
\begin{align}\label{s-harmonic}
\left[\frac{d}{dt}\cE_s\left(\frac{u+t\varphi}{|u+t\varphi|},\Omega\right)\right]_{t=0}=0,
\end{align} 
for all $\varphi\in H^s_{00}(\Omega;\R^m)\cap L^{\infty}(\Omega)$. Here
\[
H^s_{00}(\Omega;\R^m)=\{u\in H^s(\R^n;\R^m) \text{ s.t. } u=0 \text{ a.e. in } \R^n\sm\Omega\}.
\]
\end{definition}
The next proposition is just the variational formulation of the Euler-Lagrange equation of \eqref{s-harmonic}.
\begin{proposition}\label{equivalentharmonicity}
Let $u\in\hat H^s(\Omega;\R^m)$ with $|u|=1$ a.e. in $\Omega$. Then $u$ is a weak $s$-harmonic map into $\S^{m-1}$ in $\Omega$ if and only if
\begin{align}\label{harmoniceq}
\langle(-\Delta)^su,\varphi\rangle_\Omega=0,
\end{align}
for all $\varphi\in H^{s}_{00}(\Omega;\R^m)$ satisfying $\varphi(x)\in T_{v(x)}\S^{m-1}$ a.e. in $\Omega$.
\end{proposition}
The proof is the same as in \cite{Millot-Sire} but we include it for completeness.
\begin{proof}
Suppose first that $u$ is a weakly $s$-harmonic map and let $\varphi\in H^s_{00}(\Omega;\R^m)$ be such that $\varphi\cdot u$ a.e. in $\Omega$. Without loss of generality we can assume that $\varphi$ is compactly supported in $\Omega$ and that is bounded. Then we can estimate
\[
\frac{v(x)+t\varphi(x)}{|v(x)+t\varphi(x)|}=\frac{v(x)+t\varphi(x)}{\sqrt{1+t^2|\varphi^2(x)|}}=v(x)+t\varphi(x)+O(t^2),
\]
as $t\to 0$.  A direct application of the dominated convergence theorem let us deduce then 
\[
\left[\frac{d}{dt}\cE_s\left(\frac{u+t\varphi}{|u+t\varphi|},\Omega\right)\right]_{t=0}=\langle(-\Delta)^su,\varphi\rangle_\Omega,
\]
and therefore, since $u$ satisfies \eqref{s-harmonic}, $\langle(-\Delta)^su,\varphi\rangle_\Omega=0$.

Suppose now that $u\in \hat H^s(\Omega;\R^m)$ satisfies \eqref{harmoniceq} and let $\varphi\in\cD(\Omega;\R^m)$. Note that $(\varphi\cdot u)\in H^s_{00}(\Omega;\R^m)$ and therefore 
\[
\phi=\varphi-(\varphi\cdot u)\varphi
\]
belongs to $H^s_{00}(\Omega;\R^m)$. Since $|u|=1$ we have also $\phi\cdot u=0$ a.e. in $\Omega$. As before, we can rewrite
\[
\frac{v(x)+t\varphi(x)}{|v(x)+t\varphi(x)|}=\frac{v(x)+t\varphi(x)}{\sqrt{1+t^2|\varphi^2(x)|}}=v(x)+t\varphi(x)+O(t^2),
\]
as $t\to 0$. Again by dominated converge we conclude
\[
\left[\frac{d}{dt}\cE_s\left(\frac{u+t\varphi}{|u+t\varphi|},\Omega\right)\right]_{t=0}=\langle(-\Delta)^su,\varphi\rangle_\Omega,
\]
and by \eqref{harmoniceq} we conclude that \eqref{s-harmonic} holds.
\end{proof}
As in \cite{Millot-Sire} we note that thanks to the previous proposition the Euler-Lagrange equation can be rewritten as
\begin{align}\label{orthogonalityeq}
(-\Delta)^su\bot T_v\S^{m-1}, \ \text{in } H^{-s}(\Omega).
\end{align}
Note that equation \eqref{orthogonalityeq} is the nonlocal analogous of the classical harmonic map system. In the classical case ($s=1$), equation \eqref{orthogonalityeq} is equivalent to the unrestricted system
\begin{align}\label{localsystem}
(-\Delta)u=u|\nabla u|^2.
\end{align}
In the nonlocal case, we can derive a similar system of equations when the target domain is the sphere. More precisely, let $u$ be such that $|u|=1$ a.e. in $\Omega$ and $\varphi\in\cD(\Omega;\R^m)$. Thanks to Proposition \ref{equivalentharmonicity} we have
\[
\langle(-\Delta)^s u,\varphi\rangle_\Omega=\langle(-\Delta)^s u,(\varphi\cdot u)u \rangle\Omega.
\] 
Note now that since $|u|=1$ a.e. in $\Omega$ we have the following identity
\begin{align*}
(u(x)-u(y))\cdot((\varphi(x)\cdot u(x))u(x)-(\varphi(y)\cdot u(y))u(y))\\
=\frac{1}{2}|u(x)-u(y)|^2(\varphi(x)\cdot u(x)+\varphi(y)\cdot u(y)),
\end{align*}
and therefore we get
\[
\langle(-\Delta)^s u,(\varphi\cdot u)u \rangle\Omega=\frac{c_{n,s}}{2}\iint_{\Omega\times\R^n}\frac{|u(x)-u(y)|^2}{|x-y|^{n+2s}}u(x)\cdot\varphi(x)dxdy.
\]
The previous identity is equivalent then to the system
\begin{align}\label{harmonicsystem}
(-\Delta)^s u=u(x)\frac{c_{n,s}}{2}\int_{\R^n}\frac{|u(x)-u(y)|^2}{|x-y|^{n+2s}}
\end{align}
in $\cD'(\Omega)$. Note the similarity between this system and the one found in the local case \eqref{localsystem}. Furthermore note that since the constant $c_{n,s}=(1-s)c_n$ then
\[
\frac{c_{n,s}}{2}\int_{\R^n}\frac{|u(x)-u(y)|^2}{|x-y|^{n+2s}}dy\to |\nabla u(x)|^2
\] 
for smooth functions $u$. Using the notation introduced in the previous Section we can rewrite equation \eqref{harmonicsystem} as
\[
(-\Delta)^s u=u(x)\cB(u,u),
\]
and we note that as $s\to1$ we recover the classic system
\[
-\Delta u=u|\nabla u|^2.
\]
%%%%%%%%%%%%%%%%%%%%%%%%%%%%%%%%%%%%%%%%%%%%%%%%%%%%%%%%%%%%%%%%%%%%%%%%%%%%%%%%%%%%%%%%%%%%%%%%%%%%%%%%%%%%%%%%%%%%%%%%
%%%%%%%%%%%%%%%%%%%%%%%%%%%%%%%%%%%%%%%%%%%%%%%%%Nonlocal Systems%%%%%%%%%%%%%%%%%%%%%%%%%%%%%%%%%%%%%%%%%%%%%%%%%%%%%%%
%%%%%%%%%%%%%%%%%%%%%%%%%%%%%%%%%%%%%%%%%%%%%%%%%%%%%%%%%%%%%%%%%%%%%%%%%%%%%%%%%%%%%%%%%%%%%%%%%%%%%%%%%%%%%%%%%%%%%%%%

\section{Proof of Theorem \ref{maintheorem}}\label{sectionnonlocalsystems}
In this section we prove our main Theorem \ref{maintheorem} in the spirit of \cite{Caffarelli-Systems}. For the rest of the section we will denote the operator $L_k$ just by $L$. Also, since we are only concerned in the interior regularity we will assume that $\Omega=B_2(0)$. A standard covering argument then will allows us to deduce the interior regularity in any sub domain $\Omega'\subset\Omega$.
%
%In this section we are concerned with the regularity of solutions of nonlocal systems with critical nonlinearities that arise from $s$-harmonic maps. In order to go give a precise statement of our result we need some definitions. 

%Let $u,v\in \hat H^s(\Omega;\R^m)$ and define the bilinear nonlocal operator
%\begin{align}\label{bilinear}
%\cB(u,v)=\frac{c_{n,s}}{2}\int_{\R^n}\frac{(u(x)-u(y))\cdot(v(x)-v(y))}{|x-y|^{n+2s}}dy.
%\end{align}
%Note that with this notation we can write the nonlocal system associated with $s$-harmonic maps as
%\[
%(-\Delta)^s u=u\cB(u,u),
%\]
%and operator $\cB$ plays the role of the quadratic gradient in the local case. 

Recall that our right hand side is controlled by $\cB_K$ and that this operator has the same scaling as $L$ and therefore issues with regularity are expected, since it cannot be absorbed directly by the diffusion. Furthermore note that in the case of strictly smaller scaling one could proceed as in \cite{Serra} to deduce regularity via a blow up argument.

In \cite{Caffarelli-Systems}, one of the main ideas is to prove that $|u|^2$ solves a linear scalar equation. This is true due to the smallness condition on the right hand side (hypothesis (H1.1)) and the known identity $\Delta v^2=2v\Delta v+2|\nabla v|^2$. Then, thanks to the regularity theory for linear operators we can control the oscillation of the solution. More precisely, we will prove that the solution maps $B_1$ into $B_{1-\delta}$ for an appropriate choice of $\delta$.

Let us start with the following observation on the nonlocal operator $L$. Let $v:\R^n\to\R$ be a smooth bounded function, then we claim that
\[
-L v^2(x)=-2v(x)Lv(x)-2\cB_K(v,v).
\]
In fact, 
\begin{align*}
-2v(x)Lv(x)-2\cB(v,v)&=\text{P.V.}\int_{\R^n}(2v(x)(v(x)-v(y))-(v(x)-v(y))^2)K(y-x)dy \\
&=\text{P.V.}\int_{\R^n}(v^2(x)-v^2(y))K(y-x)dy\\
&=-Lv^2(x).
\end{align*}
%\begin{align*}
%-(-\Delta)^s v^2(x)&=c_{n,s}\text{P.V.}\int_\R\frac{v^2(y)-v^2(x)}{|x-y|^{n+2s}}dy\\
%&=\frac{c_{n,s}}{2}\int_\R\frac{v^2(x+y)+v^2(x-y)-2v^2(x)}{|y|^{n+2s}}dy\\
%&=\frac{c_{n,s}}{2}\int_\R\frac{(v(x+y)-v(x))^2+(v(x-y)-v(x))^2}{|y|^{n+2s}}dy\\
%&=c_{n,s}\int_\R\frac{(v(x)-v(y))^2}{|x-y|^{n+2s}}dy=\\
%&=2\cB(v,v),
%\end{align*}
which is in clear analogy with the local case. We are in shape to state our first lemma.

Another important ingredient in regularity theory is scaling. Let $u$ be a solution of \eqref{generalsystem} and assume that (H1.1), (H1.2) and (H2) hold. Let $u_{\mu,t}(x)=\mu u(tx)$, then we have that $u_{\mu,t}$ solves an analogous system 
\begin{align*}
(-\Delta)^su_{\mu,t}(x)=&\mu t^{2s}f(x,u_{\mu,t},\cB(u_{\mu,t},u_{\mu,t})t^{-2s})\\
&:=\tilde f.
\end{align*}
Hypotheses (H1.1), (H1.2) and (H2) remain valid by changing the constants accordingly,
\begin{align*}
M(u_{\mu,t})&=\mu M(u),\\
b^*(u_{\mu,t})&=\mu^2t^{2s}b^*(u),\\
b(u_{\mu,t})&=\mu t^{2s}b(u),\\
a^*(u_{\mu,t})&=\mu^2a^*(u),\\
a(u_{\mu,t})&=\mu a(u).
\end{align*}
Before we state our first lemma we stress out the fact that we will assume that the solution $u$ to \eqref{generalsystem} is smooth. This can be justified as in \cite{Caffarelli-Silvestre-ek} by a regularization procedure (Lemma 2.1).

\begin{lemma}\label{contractmap}
Let $u$ be a weak solution to \eqref{generalsystem} in $B_2(0)$ satisfying hypotheses (H1.1), (H1.2) and (H2). Assume also that $a=1$, $b=0=b^*$ and that $1/2(a^*+M)=l<1$. Then there exists a constant $0<\delta(l)<1$ such that 
\[
u(B_1(0))\subset B_{M(1-\delta)(\delta\bar u)},
\]
where
\[
\bar u=\frac{1}{|B_1|}\int_{B_1}u dx=\fint_{B_1}udx.
\]
Furthermore $\delta$ is monotone decreasing in $l$.
\end{lemma}
\begin{proof}
As mentioned before the strategy revolves in using $|u|^2$ as a supersolution of a linear scalar equation. 

First note that
\begin{align*}
L(|u|^2)&=L\left(\sum_i u_i^2\right)\\
&=\sum_i\left(2u_iLu_i+2\cB_K(u_i,u_i)\right)\\
&= -2u\cdot f(x,u,\cB_K(u,u))+2\cB_K(u,u)\\
&\geq 2(1-l)\cB_K(u,u).
\end{align*}
Let $\rho\in\R^m$ with $|\rho|\leq 1-l$ and note that
\begin{align*}
-L(\rho\cdot u)&=\rho\cdot-Lu\\
&=\rho\cdot f(x,u,\cB_K(u,u))\\
&\leq|\rho||\cB_K(u,u)|.
\end{align*}
which leads to \[
-L\left(\frac{1}{2}|u|^2+\rho\cdot u\right)\geq 0
\]
Recall now that $u$ is bounded by $M$, therefore  
\[
h(x)=\frac{1}{2}M^2+(1-l)M-\frac{1}{2}|u|^2-\rho\cdot u,
\]
is nonnegative and furthermore satisfies $-L h\geq 0$, therefore there exists a constant $C=C(\lambda, \Lambda)$ (independent of $s$, see for example \cite{Caffarelli-Silvestre}) such that for all $x,y\in B_1$ we have
\[
h(y)\leq C h(x).
\]
Taking average we conclude then that for all $x\in B_1$ 
\begin{align*}
h(x)&\geq \frac{1}{C}\bar h\\
&\geq c_1 ((1-l)M-\rho\cdot\bar u),
\end{align*}
or equivalently
\begin{align}\label{control1}
\frac{M-|u|^2}{2}+(1-l)M-\rho\cdot u\geq c_1 ((1-l)M-\rho\cdot\bar u). 
\end{align}
Now we are in position to prove the conclusion of the lemma. For this, note that inequality \eqref{control1} allows us to control $u$ and not only $|u|^2$. In fact, take $\rho$ in the direction of $u$ with $|\rho|=1-l$, denote by $\theta$ the angle between $u$ and $\bar u$ and let $r=|u|/M$. With this selection of parameters we get from \eqref{control1}
\[
M(1-r)\left(\frac{1}{2}(M+|u|)+(1-l)\right)\geq Mc_1(1-l)\left(1-\frac{|\bar u|}{M}\cos\theta\right),
\]
which gives us the control on $r$
\begin{align}\label{control2}
1-r\geq c_2\left(1-\frac{|\bar u|}{M}\cos\theta\right).
\end{align}
Note that thanks to the hypothesis $1/2(a^*+M)=l<1$ the constant $c_2$ is uniformly bounded independent on $M$, which we can assume without any loss of generality smaller than 1 ($c_2<1$). Therefore by multiplying \eqref{control2} by $r$ and adding afterwards $1-r$ we arrive to 
\[
1-r^2\geq c_2\left(1-r\frac{|\bar u|}{M}\cos\theta\right),
\]
which is equivalent to 
\[
r^2-c_2r\frac{|\bar u|}{M}\cos\theta\leq 1-c_2.
\]
Note now that $\bar u/M\leq 1$ therefore from the previous inequality we get
\[
r^2-c_2r\frac{|\bar u|}{M}\cos\theta+\left(\frac{1}{2}c_2\frac{\bar u}{M}\right)^2\leq 1-c_2 +\left(\frac{1}{2}c_2\right)^2,
\]
and by picking $\delta=1/2c_2$ we conclude
\[
|u-\delta\bar u|^2\leq M^2(1-\delta)^2.
\]
which finishes the proof.
\end{proof}
The previous lemma states that $u$ maps $B_1(0)$ to a ball of strictly smaller radius and center shifted toward $\bar u$. This result turns out to control the oscillation of the function. Note that the key ingredient is the fact that we can simplify the system to the study of a scalar linear equation.

A direct consequence of the previous lemma is the following corollary.
\begin{corollary}\label{corocontract}
Let $u$ be as in Lemma \ref{contractmap}. Then there exist a sequence of points $\{\rho_k\}$ and radii $\{M_k\}$ such that 
\begin{itemize}
\item[i.-] $M_k\leq M(1-\delta)^k$.
\item[ii.-] $|\rho_k|+M_k\leq M$.
\item[iii.-] $u(B_{2^{-k}}(0))\subset B_{M_k}(\rho_k) $.
\end{itemize}
\end{corollary}
\begin{proof}
We proceed by induction on $k$. Note the case $k=0$ is just Lemma \ref{contractmap} with $\rho_0=\delta\bar u$ and $M_0=M(1-\delta)$. Let $u_k=u(2^{-k}x)-\rho_k$ and assume the result holds. In order to apply Lemma \ref{contractmap} to $u_k$ in $B_1$ we first note that 
\[
\|u_k\|_{L^\infty(B_1(0))}\leq M_k+|\rho_k|\leq M.
\]
Furthermore note that thanks to the bounds of $f$ and the scaling properties of fractional Laplacian $u_k$ solves 
\[
-L u_k=\bar f(x,u_k,\cB_K(u_k,u_k)),
\]
and $\bar f$ satisfies hypotheses (H1.1) and (H1.2) with constants $a^*_k:=a^*(u_k)=a^*$ (see the scaling remark before Lemma \ref{contractmap}). Therefore we can apply Lemma \ref{contractmap} to $u_k$, which finishes the proof by letting $M_{k+1}=M_k(1-\delta)$, $\rho_{k+1}=\rho_k+\delta\bar u_k$.
\end{proof}
Note that, as in the local case, if $\bar u<M$, then there is no need to shift the center of the ball to get an improvement on the $L^\infty$ norm of $u$. Furthermore instead of asking the structural condition $1/2(a^*+M)<1$ in order to apply Corollary \ref{corocontract} we just need $a^*<1$ and 
\[
a^*+\liminf\limits_{r\to0}\fint_{B_r} u<2.
\] 
We need to lift now the extra assumption of Lemma \ref{contractmap}. The following lemma deals with the case of nontrivial $b$ and $b^*$. Without loss of generality let us assume that $b=\max\{b,b^*\}$.
\begin{lemma}\label{contractmapb}
Let $u$ be a weak solution to \eqref{generalsystem} in $B_2(0)$ satisfying hypotheses (H1.1), (H1.2), (H2) and assume also that $a=1$. Then there exists a constant $\tau$ such that 
\[
u(B_1(0))\subset B_{M(1-\delta)+\tau b}(\delta\bar u),
\]
where $\delta$ is the same from Lemma \ref{contractmap}.
\end{lemma}
Note that the main difference between Lemma \ref{contractmap} and Lemma \ref{contractmapb} is that in the latter we have to take into the account the action of the nontrivial factors.
\begin{proof}
We will proceed as in Lemma \ref{contractmap}, but first we need to add a correcting factor to the function $h$. Let $v$ be the solution to 
\begin{align*}
\begin{cases}
\displaystyle -L v=-1 &\text{ in } B_2,\\
\displaystyle v=0 &\text{ in } \R^n\sm\Omega.
\end{cases}
\end{align*}
%We point out that $v$ has a an explicit formula, but we have left it written as the solution to the nonlocal PDE. When we extend our results to more general operators in the next section, the correcting factor will be the solution to the scalar Dirichlet problem for the associated nonlocal operator.

Note that $v\leq 0$ by the maximum principle and that it is universally bounded in $B_1$
\[
\|v\|_{L^\infty(B_1(0))}\leq L,
\]
for some $L$.

Define now
\[
h(x)=\frac{1}{2}M^2+(1-l)M-\frac{1}{2}|u|^2-\rho\cdot u-2bv,
\]
which, as in Lemma \ref{contractmap}, is a nonnegative function solving
\[
(-\Delta)^sh\geq 0 \ \text{in } \Omega.
\]
Applying the Harnack inequality to $h$ and taking average we deduce as before
\begin{align*}
h(x)&\geq \frac{1}{C}\bar h\\
&\geq c_1 [(1-l)M-\rho\cdot\bar u-2b\bar v]\\
&\geq c_1 [(1-l)M-\rho\cdot\bar u-2bL].
\end{align*}
Recall now that $v\leq 0$, therefore we have 
\[
\frac{1}{2}M^2+(1-l)M-\frac{1}{2}|u|^2-\rho\cdot u\geq h(x),
\]
and so rearranging the terms as in Lemma \ref{contractmap} we deduce
\begin{align*}
\frac{M-|u|^2}{2}+(1-l)M-\rho\cdot u+2c_1Lb\geq c_1 ((1-l)M-\rho\cdot\bar u). 
\end{align*}
Take $\rho$ in the direction of $u$ with $|\rho|=1-l$, denote by $\theta$ the angle between $u$ and $\bar u$ and let $r=|u|/M$. From the previous inequality we deduce
\[
M(1-r)\left(\frac{1}{2}(M+|u|)+(1-l)\right)+2c_1Lb\geq Mc_1(1-l)\left(1-\frac{|\bar u|}{M}\cos\theta\right).
\]
At this point we can proceed as in the proof of Lemma \ref{contractmap} to deduce the desired conclusion. 
\end{proof}
Since the coefficients $b, b^*$ are nontrivial we note that we no longer have the inclusion of 
\[
B_{M(1-\delta)+\tau b}(\delta\bar u)\subset B_1(0).
\]
This inclusion was crucial in order to prove Corollary \ref{corocontract}, since it allowed us to control
\[
M_k+a^*_k\leq M+a^*,
\]
where $a^*_k$ stands for the corresponding constant $a^*$ associated to $u_k(x)=u(2^{-k}x)-\rho_k$. In order to control now the constants we note that $u_k$ solves the same system \eqref{generalsystem} with the appropriate constant (see the scaling remark at the beginning of the section)
\[
b_k:=b(u_k)\leq 2^{-2sk}(1+|\rho_k|)b,
\]
which will be sufficient to prove that the balls remain within $M+(1-l)$ from the original one. 

We now iterate Lemma \ref{contractmapb} as we did in Corollary \ref{corocontract}. As remarked before we have to take into the account that, a priori, the balls are not contained in the previous one.
\begin{corollary}\label{corocontractb}
Let $u$ as in Lemma \ref{contractmapb}. Then there exists a constant $d=d(l,b)$ and sequence of points vectors $\{\rho_k\}\subset B_M(0)$ and radii $\{M_k\}$ such that
\begin{itemize}
\item[i.-]$M_k\leq M(1-\frac{1}{2}\delta(\frac{1}{2}(1+l)))$
\item[ii.-] $|\rho_k|+M_k\leq M+(1-l)\sum_{i=1}^k2^{-si}$.
\item[iii.-] $u(B_{2^{-(k+d)}})\subset B_{M_k}(\rho_k)$.
\end{itemize}
\end{corollary}
\begin{proof}
Without loss of generality we can assume that $\tau\geq 1$, $M>1/2$ and $\delta<1/2^s$ and that $b=\max\{b,b^*\}$. 

Let us pick $d$ large enough so that
\[
2^{-d}b\tau(1+M)\leq\min\left\{1-l,\frac{2^s-1}{2^s}M\delta\right\}.
\]
Note in particular that 
\[
\frac{2^s-1}{2^s}\leq \frac{1}{2}.
\]
We will prove the result by induction. For the initial step we apply Lemma \ref{contractmapb} to the function $u_0(x)=u(2^{-d}x)$ to get that $u_0(B_1)\subset B_{M(1-\delta)+\tau b_0(\delta \bar u_0)}$. We translate the inclusion to 
\[
u_0(B_1(0))=u(B_{2^{-d}}(0))\subset B_{M(1-\delta)+\tau b_0}(\delta \bar u_0).
\] 
Define $\rho_0=\delta \bar u_0$ and $M_0=M(1-\delta)+\tau b_0$ and let us check the conditions. Since $M(u_0)=M$ then we have that $\rho_0\in B_{\delta M}\subset B_M$. Furthermore since $b(u_0)=2^{-2sd}b$ we get that ($s\in[1/2,1)$)
\begin{align*}
M_0=2^{-2sd}b\tau+(1-\delta)M&\leq 2^{-d}b\tau+(1-\delta)M\\
&\leq \min\{1-l,1/2M\delta\} +(1-\delta)M\\
&\leq M.
\end{align*}
Finally note that
\begin{align*}
\rho_0+M_0&\leq\delta M+\min\{1-l,1/2M\delta\} +(1-\delta)M\\
&\leq M+(1-l),
\end{align*}
which finishes the initial step. 

Assume now that the result is valid for $k$ and define $u_{k+1}(x)=u(2^{-(k+1+d)}x)=u_0(2^{-(k+1)}x)$. Note that $a^*_{k+1}:=a^*(u_{k+1})=a^*$ and since $u_{k+1}(B_2(0))=u_k(B_1(0))\subset B_{M_k}(|\rho_k|)$ therefore we have
\begin{align*}
a^*_{k+1}+M(u_{k+1})&\leq a^*+M_k+|\rho_k|\\
&\leq a^*+M+(1-l)\leq 1+l.
\end{align*}
Furthermore we also have $b(u_k)=2^{-2s(k+d)}b\leq 2^{-2sk}\min\{1-l,\frac{1}{2}M\delta\}$. Now, Lemma \ref{contractmapb} applied to $u_{k+1}$ gives us
\[
u_{k+1}(B_1(0))\subset B_{M_k(1-\delta)+\tau b_k}(\delta \bar u_{k+1}).
\]
Also Lemma \ref{contractmapb} gives us a point $\rho_{k+1}$ lying in the segment $\rho_k$ and $\delta \bar u_k$ and hence in $B_M$ and a radius
\begin{align*}
M_{k+1}&\leq M_k(1-\delta)+\tau b_k\\
&\leq M\left(1-\frac{1}{2^s}\delta\left[\frac{1}{2}(1+l)\right]\right)^k(1-\delta)+2^{-2sk}M\frac{2^s-1}{2^s}M\delta\\
&\leq M\left(1-\frac{1}{2^s}\delta\right)^k(1-\delta)+2^{-2sk}M\frac{2^s-1}{2^s}M\delta.
\end{align*}
Since $\delta\leq 1/2^s$ we have that $1-\delta/2^s\geq 2^{-2s}$, 
\[
2^{-2sk}\leq \left((1-\frac{1}{2^s}\delta\right)^k
\]
Since
\[
\frac{2^s-1}{2^s}\delta\leq \left(1-\frac{1}{2^s}\right).
\]
we conclude then
\begin{align*}
M_{k+1}&\leq M\left(1-\frac{1}{2^s}\delta\right)^k(1-\delta)+\left(1-\frac{1}{2^s}\delta\right)^k \frac{2^s-1}{2^s}M\delta\\
&\leq M\left(1-\frac{1}{2^s}\delta\right)^{k+1}.
\end{align*}
Finally we estimate
\begin{align*}
|\rho_{k+1}|+M_{k+1}&\leq |\rho_k|+\delta\bar u_{k+1}+M_k(1-\delta)+\tau b_k\\
&\leq|\rho_k|+M_k+\tau b_k\\
&\leq M+(1-l)\sum_{i=1}^k2^{-si}+\tau b2^{-2s(k+1+d)}\\
&\leq M+(1-l)\sum_{i=1}^{k+1}2^{-si},
\end{align*}
which finishes the proof. 
\end{proof}
Since the oscillation decreases at every step the conclusion of the Theorem \ref{maintheorem} now follows in a standard way. 

\begin{remark}
As pointed out before, we can replace the structural condition
\[
a^*+aM\leq 2,
\]
by $a\leq 1$ and
\[
a^*+a\liminf\limits_{r\to 0}\left| \fint_{B_r(x_0)}u\right|\leq 2
\]
to conclude that $u$ is H\"older continuous in a neighborhood of $x_0$.
\end{remark}

%%%%%%%%%%%%%%%%%%%%%%%%%%%%%%%%%%%%%%%%%%%%%%%%%%%%%%%%%%%%%%%%%%%%%%%%%%%%%%%%%%%%%%%%%%%%%%%%%%%%%%%%%%%%%%%%%%%%%%%%
%%%%%%%%%%%%%%%%%%%%%%%%%%%%%%%%%%%%%%%%%%%%%%%%%Extension and counterexample%%%%%%%%%%%%%%%%%%%%%%%%%%%%%%%%%%%%%%%%%%%
%%%%%%%%%%%%%%%%%%%%%%%%%%%%%%%%%%%%%%%%%%%%%%%%%%%%%%%%%%%%%%%%%%%%%%%%%%%%%%%%%%%%%%%%%%%%%%%%%%%%%%%%%%%%%%%%%%%%%%%%
\section{A non regular example when $aM+a^*=2$}\label{sectionextension}
%%%%%%%%%%%%%%%%%%%%%%%%%%%%%%%%%%%%%%%%%%%%%%%%%%%%%%%%%%%%%%%%%%%%%%%%%%%%%%%%%%%%%%%%%%%%%%%%%%%%%%%%%%%%%%%%%%%%%%%%
%%%%%%%%%%%%%%%%%%%%%%%%%%%%%%%%%%%%%%%%%%%%%%%%%%%%%%%%%%%%%BIBLIOGRAPHY%%%%%%%%%%%%%%%%%%%%%%%%%%%%%%%%%%%%%%%%%%%%%%%
%%%%%%%%%%%%%%%%%%%%%%%%%%%%%%%%%%%%%%%%%%%%%%%%%%%%%%%%%%%%%%%%%%%%%%%%%%%%%%%%%%%%%%%%%%%%%%%%%%%%%%%%%%%%%%%%%%%%%%%%
In this section we provide an example of a non regular solution when the structural condition is not satisfied. In the local case the harmonic map to the unitary sphere provides the non smooth solution $\Phi=x/|x|:\R^n\to\S^{n-1}$ to 
\[
-\Delta \Phi=\Phi|\nabla \Phi|^2,
\]
for general dimensions $n$. In the particular case $n=1$ the function $\Phi$ also solves
\[
(-\Delta)^s\Phi=\Phi\cB(\Phi,\Phi),
\]
for all $s\in (0,1)$. In fact, since $\Phi$ is just the sign function we have for $x,y\in\R\sm\{0\}$
\begin{align*}
\Phi(x)(\Phi(x)-\Phi(y))^2&=\Phi(x)(\Phi^2(x)-2\Phi(x)\Phi^(y)+\Phi^2(y))\\
&=2\Phi(x)(1-\Phi(x)\Phi(y))\\
&=2\Phi(x)(\Phi(x)-\Phi(y)),
\end{align*}
where we used the fact that $\Phi^2=1$. Therefore the following formal computation
\begin{align*}
\cB(\Phi,\Phi)&=\frac{c_{n,s}}{2}\int_{\R^n}\frac{\Phi(x)-\Phi(y))^2}{|x-y|^{1+2s}}dy\\
&=\frac{c_{n,s}}{2}\int_{\R^n}\frac{2(\Phi(x)-\Phi(y))}{|x-y|^{1+2s}}dy\\
&=(-\Delta)^s\Phi(x),
\end{align*}
concludes the claim.

In this case, we have a non smooth solution to the system in the case $aM+a^*=2$. We point out that the previous formal computation can be justified by taking a smooth approximation of the sign function $\Phi_n$ such that $\Phi_n=\Phi$ for $x\in\R\sm(-1/n,1/n)$.

We point out that for general dimensions $x/|x|$ fails to solves the fractional harmonic system. This is mainly due to the fact that for dimensions greater than 1 there is a nonlocal interaction with the coordinates and therefore the projection to the sphere fails to solve the nonlocal system. A counterexample for general dimensions is still open for the nonlocal case.

Let us give a brief remark on the passage to the limit as $s\to1$. As noted in \cite{Caffarelli-Silvestre} we have that for a smooth function $v$
\[
\lim\limits_{s\to1}\frac{c_{n,s}}{2}\int_{\R^n}\frac{v(x+y)+v(x-y)-2v(x)}{|y|^{n+2s}}=\Delta v(x)
\]
and therefore changing $z=Ay$ we deduce
\[
\lim\limits_{s\to1}\frac{c_{n,s}}{2}\int_{\R^n}\frac{v(x+y)+v(x-y)-2v(x)}{\text{det}A|A^{-1}z|^{n+2s}}dz=\sum a_{ij}v_{ij}(x),
\]
where $a_{ij}=AA^t$. With this fact we can recover a priori estimates for (viscosity) solutions to systems of the form
\[
Lu=f(x,u,\nabla u),
\]
where $L=\sum a_{ij}\partial_{ij}$.

On the other hand, given a kernel $K\in\cK$ the operator $L_K$
\[
L_K=\int_{\R^n}(v(y)-v(x))K(y-x)dy
\]
is the Euler Lagrange equation of the energy integral
\begin{align*}
\cE_{K,s}(u)=&\iint_{\R^{2n}}|u(x)-u(y)|^2K(x-y)dxdy.
\end{align*} 
Since the associated energy converges to the classical Dirichlet energy, weak solutions to the fractional equations will converge to classic divergence type equations.

The previous assertion still holds for more general operators of the form
\[
Lu=\int_{\R^n}(u(y)-u(x))K(x,y)dy
\]
under symmetry assumptions $K(x,y)=K(y,x)$ and satisfying bounds like \eqref{boundK} uniformly in $x$. The associated energy here is simply given by
\[
\iint_{\R^{2n}}|u(x)-u(y)|^2|x-y|^{n+2s}K(x,y)dxdy.
\]
With this in mind and since Theorem \ref{maintheorem} is stable in $s$ we recover the a priori H\"older estimates for (weak) solutions as in \cite{Caffarelli-Systems}.

\noindent {\bf Acknowledgements.} 

L. Caffarelli was partially supported by NSF grant DMS-1160802 and NSF grant DMS-1540162

G. D\'avila was partially supported by Fondecyt grant No. 11150880.

\end{document}